\documentclass[12pt]{amsart}
\usepackage{amstext,amsfonts,amssymb,amscd,amsbsy,amsmath,verbatim, mathrsfs, fullpage}
\usepackage[alphabetic,abbrev,lite]{amsrefs} % for bibliography 
\usepackage{ifthen,tikz}
\usepackage{color}
\usepackage{amsthm}
\usepackage{latexsym}
\usepackage[all]{xy}
\usepackage{enumerate}
\usepackage{mathtools}
\numberwithin{equation}{section}

\newtheorem{lemma}[equation]{Lemma}

\newtheorem{thm}[equation]{Theorem}

\theoremstyle{definition}
\newtheorem{defn}[equation]{Definition}

\newtheorem{example}[equation]{Example}

\theoremstyle{remark}
\newtheorem{remark}[equation]{Remark}

\newtheorem*{claim*}{Claim}
\newtheorem*{case*}{Case}

\newcommand{\cC}{\mathcal{C}}

\newcommand{\PP}{\mathbb P}

\newcommand{\ZZ}{\mathbb Z}

\newcommand{\pdim}{\operatorname{pdim}}

\newcommand{\Spec}{\operatorname{Spec}}

\newcommand{\Pic}{\operatorname{Pic}}

\newcommand{\Hom}{\operatorname{Hom}} %done

\newcommand{\cO}{{\mathcal O}}

\newcommand{\rank}{\operatorname{rank}}

\newcommand{\codim}{\operatorname{codim}}
\newcommand{\depth}{\operatorname{depth}}
\newcommand{\defi}[1]{\textsf{#1}} % for defined terms

\newcommand{\beq}{\begin{displaymath}}
\newcommand{\eeq}{\end{displaymath}}

\def\nc{\newcommand}
\def\on{\operatorname}
\nc{\Q}{\mathbb{Q}}
\nc{\RR}{\mathbf{R}}
\nc{\LL}{\mathbf{L}}
\nc{\xra}{\xrightarrow}
\nc{\xla}{\xleftarrow}
\def\a{\alpha}

\def\DM{\operatorname{DM}}

\nc{\into}{\hookrightarrow}
\nc{\onto}{\twoheadrightarrow}
\nc{\OO}{\mathcal{O}}
\nc{\Z}{\mathbb{Z}}
\nc{\cA}{\mathcal{A}}
\nc{\w}{\widehat}
\nc{\End}{\on{End}}
\nc{\res}{\frac{1}{x_0x_1}}
\nc{\tF}{\widetilde{F}}
\nc{\tG}{\widetilde{G}}
\nc{\tf}{\widetilde{f}}
\nc{\Com}{\on{Com}}

\nc{\G}{\mathbb{G}}
\nc{\cG}{\mathcal{G}}
\nc{\cE}{\mathcal E}
\nc{\cF}{\mathcal F}
\nc{\cR}{\mathcal R}
\nc{\cD}{\mathcal D}
\nc{\cB}{\mathcal B}
\nc{\cT}{\mathcal T}
\nc{\cL}{\mathcal L}

\nc{\bM}{\mathbf M}
\nc{\bN}{\mathbf N}
\nc{\U}{\mathbf U}
\nc{\BM}{\mathbf B \mathbf M}
\nc{\Dsg}{\on{D}_{\on{sg}}}
\nc{\fC}{\mathcal{C}}
\nc{\fG}{\mathcal{G}}
\nc{\N}{\mathbb{N}}

\nc{\del}{\partial}
\nc{\cone}{\on{cone}}
\nc{\D}{\on{D}_{\on{diff}}}
\nc{\DMb}{\on{D}^b_{\DM}}
\nc{\Db}{\on{D}^{\on{b}}}
\nc{\Kb}{\on{K}^{\on{b}}}
\nc{\fm}{\mathfrak{m}}
\nc{\Flag}{\on{Flag}}
\nc{\DMmin}{\DM_{\on{min}}}
\nc{\Ddiff}{\on{D}_{\on{diff}}}
\nc{\Dbdiff}{\on{D}^\on{b}_{\on{diff}}}
\nc{\wO}{\widehat{\OO}}
\nc{\wT}{\widehat{T}}
\nc{\from}{\leftarrow}
\nc{\wLL}{\widetilde{\LL}}
\nc{\augCech}{\widetilde{\cC}}
\nc{\Fold}{\on{Fold}}
\nc{\Ext}{\on{Ext}}
\nc{\FF}{\mathbf{F}}
\nc{\Comper}{\Com_{\on{per}}}
\nc{\Unfold}{\on{Unfold}}
\nc{\intHom}{\underline{\Hom}}
\nc{\Ex}{\on{Ex}}
\nc{\tg}{\widetilde{g}}

\def\b{\beta}
\nc{\B}{\mathcal{B}}
\nc{\K}{\mathcal{K}}
\nc{\kos}{\on{Kos}}
\nc{\Perf}{\on{Perf}}
\nc{\tR}{\widetilde{\cR}}
\nc{\X}{\mathcal{X}}
\nc{\Cl}{\on{Cl}}
\nc{\fU}{\mathcal{U}}
\nc{\bU}{\mathbf U}

\nc{\st}{\on{st}}

\nc{\coh}{\on{coh}}

\def\D{\mathcal{D}}

\nc{\tU}{\U}
\nc{\bC}{\mathbf{C}}
\nc{\aux}{\on{aux}}

\def\phi{\varphi}

\nc\Dsing{\on{D}^{\on{sing}}}

\newcommand{\Manoa}{M\=anoa}

\newcommand{\Hawaii}{Hawai\kern.05em`\kern.05em\relax i}
\newcommand{\UHM}{University of \Hawaii \ at \Manoa}
\def\MR#1{}

\title{A short proof of the Hanlon-Hicks-Lazarev Theorem}

\author{Michael K. Brown}
\author{Daniel Erman}
\thanks{The second author was supported by NSF grant 
DMS-2200469.}

\newcommand{\Addresses}{{
	\vskip\baselineskip
  	\footnotesize
  	\noindent \textsc{Department of Mathematics and Statistics, Auburn University, Auburn, AL} \par\nopagebreak
	\noindent \textit{E-mail address:} \texttt{mkb0096@auburn.edu}
	\vskip\baselineskip
	\noindent \textsc{Department of Mathematics, \UHM, Honolulu, HI} \par\nopagebreak
	\noindent \textit{E-mail address:} \texttt{erman@hawaii.edu}
  }}

\date{\today}

\begin{document}

\begin{abstract}
We give a short, new proof of a recent result of Hanlon-Hicks-Lazarev about toric varieties.  As in their work, this leads to a proof of a conjecture of Berkesch-Erman-Smith on virtual resolutions and to a resolution of the diagonal in the simplicial case. 
%We also discuss a novel application concerning virtual resolutions that ``see'' the birational geometry. \michael{wordsmithed last line of abstract slightly}
\end{abstract}

\maketitle

\section{Main result}
 We give a short, new proof of a recent result of Hanlon-Hicks-Lazarev about toric varieties and their multigraded Cox rings.  Throughout, we let $X$ be a simplicial, projective toric variety over an  algebraically closed field $k$ with $\Cl(X)$-graded Cox ring $S$. Our main result (Theorem~\ref{thm:virtual hochster}) was first proven in~\cite{HHL}, but our proof is independent from their methods.
Our approach is more algebraic and simpler, while their approach is more explicit and connects to a wider range of topics, including symplectic geometry and homological mirror symmetry.  See also the work of Favero-Huang~\cite{FH}, which was completed simultaneously with \cite{HHL} and whose main results coincide with some of Hanlon-Hicks-Lazarev's. %the results of \cite{HHL}.
%whose main results overlap with some of those in~\cite{HHL}.\daniel{I want to give a bit more credit to Favero-Huang but am not sure how.} \michael{I see what you mean, but I played with this for awhile, and I'm not sure how to improve it either. I think it's okay as is. We also give them more credit in two different spots later on Section 1.}

Our interest in these topics begins with a program to extend results on syzygies to multigraded or toric settings.  The basic perspective, introduced by Berkesch-Erman-Smith in~\cite{BES}, is that many classical results about minimal free resolutions will have strong analogues in the toric setting, as long as one replaces minimal free resolutions with the more flexible notion of a virtual resolution.
\begin{defn}
\label{defn:virtual}
Let $M$ be a finitely generated $\Cl(X)$-graded $S$-module.  A \defi{virtual resolution} of $M$ is a free complex $F_\bullet$ of $S$-modules such that there is a quasi-isomorphism $\widetilde{F_\bullet} \xra{\simeq} \widetilde{M}$ of complexes of $\OO_X$-modules.\footnote{If $X$ is smooth, then $\widetilde{F_\bullet}$ consists of sums of line bundles and is sometimes called a line bundle resolution.  See Remark~\ref{rmk:not line bundles} regarding the simplicial case.}
\end{defn}

\noindent The following is a consequence of Hanlon-Hicks-Lazarev's main result~\cite[Theorem~A]{HHL}.

\begin{thm}\label{thm:virtual hochster}
Let $Y$ be a normal toric variety and $Y \into X$ a closed immersion that is a toric morphism \cite[Definition 3.3.3]{CLS}. Denote by $I$ the defining ideal of $Y \into X$ (Definition~\ref{def:ideal}).  The $S$-module $S/I$ admits a virtual resolution of length $\codim(Y\subseteq X)$.
\end{thm}

And here is our short proof of Theorem~\ref{thm:virtual hochster}. The proof relies on some elementary facts about toric varieties that we recall in Lemma~\ref{lem:technical} below.

\begin{proof}[Proof of Theorem~\ref{thm:virtual hochster}]
The Cox ring $S$ of $X$ is positively $\Cl(X)$-graded~\cite[Definition~A.1, Example~A.2]{BE}, and so we may consider $\Cl(X)$-graded minimal free resolutions of $\Cl(X)$-graded $S$-modules. Let $R$ be the normalization of $S/I$ and $F_\bullet$ the minimal free resolution of $R$ as an $S$-module.  Since $Y$ is normal, $\widetilde{R}=\cO_Y$ as a sheaf on $X$, and so $F_\bullet$ is a virtual resolution of $S/I_Y$. By Lemma~\ref{lem:technical}(1) and~\cite[Theorem 1.1.17 and Proposition 1.3.8]{CLS}, the ring $R$ is a product of affine semigroup rings of the same dimension. Hochster's Theorem therefore implies that each component of $R$ is a Cohen-Macaulay ring~\cite[Theorem 1]{hochster}. It follows that $R$ is also a Cohen-Macaulay $S/I$-module: indeed, we have $\dim(R) = \dim(S/I)$, and since $R$ is a finitely generated $S/I$-module~\cite[Theorem 4.14]{eisenbudbook}, any system of parameters on $S/I$ is a system of parameters on each component of $R$ and hence a regular sequence. The length of $F_\bullet$ is the projective dimension of $R$, which, by the Auslander-Buchsbaum formula~\cite[Theorem 19.9]{eisenbudbook}, is equal to $\depth_S(S) - \depth_S(R) = \dim(S) - \dim(S/I)$ (while the version of the Auslander-Buchsbaum formula we cite pertains to local rings, the desired result for the polynomial ring $S$ follows by~\cite[Proposition 1.5.15]{BH}). Lemma~\ref{lem:technical}(2) therefore implies that the length of $F$ is equal to $\codim(Y\subseteq X)$.
\end{proof}

%Let us briefly 
We now describe
%some 
applications of Theorem~\ref{thm:virtual hochster} and their history.  For a fuller discussion, see~\cite[\S1]{HHL}.  We start with a 
%following 
special case, first proven by Hanlon-Hicks-Lazarev:
\begin{thm}[\cite{HHL} Corollary B]\label{thm:virtual diagonal}
The coordinate ring of the diagonal embedding $X\subseteq X\times X$ admits a virtual resolution of length $\dim X$.
\end{thm}
Special cases of Theorem~\ref{thm:virtual diagonal} were studied in~\cite{BE,brown-sayrafi,canonaco}, and ~\cite{BPS,anderson} study closely related questions.
It was known that this result would immediately yield proofs of two conjectures that also had received independent interest. The first conjecture is due to Berkesch-Erman-Smith~\cite[Question 1.3]{BES} and was proven by Hanlon-Hicks-Lazarev:

\begin{thm}[\cite{HHL} Corollary C]\label{thm:virtual syzygy}
Any module $M$ as in Definition~\ref{defn:virtual} has a virtual resolution of length  $\leq \dim X$.
\end{thm}

Hilbert's Syzygy Theorem gives a bound of $\dim S=\dim X + \rank \Cl(X)$; Theorem~\ref{thm:virtual syzygy} implies that the added flexibility of virtual resolutions allows for significantly shorter resolutions, especially when $\rank \Cl(X)$ is large.  See~\cite{BES,HNV,berkesch-klein-loper-yang} and elsewhere for many examples of this phenomenon.   
%In~\cite{HHL}, Theorem~\ref{thm:virtual diagonal} is stated for smooth varieties, but as we will see, the basic ideas easily extend to the simplicial case.  
Prior to~\cite{HHL}, Theorem~\ref{thm:virtual syzygy} had been proven in several special cases: when $\rank \Pic(X)=1$ it essentially follows from Hilbert's Syzygy Theorem; for products of projective spaces it was shown in~\cite[Theorem~1.2]{BES} (see also \cite[Corollary~1.14]{EES}); Yang proved it for any monomial ideal in the Cox ring of a smooth toric variety~\cite{yang}; and Brown-Sayrafi proved it for smooth projective toric varieties of Picard rank 2~\cite{brown-sayrafi}.

The second conjecture, due to Orlov, is the special case of \cite[Conjecture 10]{orlov} for toric varieties. This was first proven by Favero-Huang in \cite[Theorem~1.2]{FH}, and independently and essentially simultaneously in ~\cite[Corollary E]{HHL}.
\begin{thm}\label{thm:rouquier}
%Let $X$ be a normal toric variety.  
The Rouquier dimension of $D^b(X)$ equals $\dim X$.
\end{thm}

Special cases of Theorem~\ref{thm:rouquier} had been established in~\cite{BC, BF, BDM, BFK} before Favero-Huang and Hanlon-Hicks-Lazarev proved it in general. The full version of Orlov's Conjecture states that Theorem~\ref{thm:rouquier} extends to any smooth quasi-projective variety; see \cite[\S 1.2]{BC} for a list of known cases of this conjecture. 

Theorem~\ref{thm:virtual hochster} easily implies Theorems~\ref{thm:virtual diagonal}, ~\ref{thm:virtual syzygy} and~\ref{thm:rouquier}. To prove Theorem~\ref{thm:virtual diagonal}, observe that the diagonal $X \subseteq X \times X$ satisfies the conditions of Theorem~\ref{thm:virtual hochster}. To prove Theorem~\ref{thm:virtual syzygy}, one can simply follow the method of~\cite[Proof of Theorem~1.2]{BES}.  For Theorem~\ref{thm:rouquier}, one can use standard techniques on derived categories; see, e.g., the proof of ~\cite[Corollary E]{HHL}.

\medskip

Our proof of Theorem~\ref{thm:virtual hochster} is quite simple, perhaps embarrassingly so given the prior partial results on these questions cited above.
%~\cite[Question 1.3]{BES} and the toric case of~\cite[Conjecture 10]{orlov}. 
It is not yet clear how to compare our resolutions to those obtained in~\cite{HHL}, but we believe that the two constructions agree in the case of Theorem~\ref{thm:virtual diagonal}. Their work gives a 
%marvelously 
creative perspective on building these resolutions, drawing motivation from the symplectic side of the mirror symmetry functor and involving a wide array of ideas.%, including discrete morse theory, quiver diagrams, and more.
\footnote{In a different direction, we refer to Borisov's work \cite{borisov} for an alternative proof of Hochster's Theorem \cite[Theorem 1]{hochster}---the main ingredient of our proof of Theorem~\ref{thm:virtual hochster}---and an explanation of how the techniques used there relate to mirror symmetry.} The resolutions they obtain are quite explicit; indeed, their resolution of the diagonal yields a canonical generating set for the derived category of any normal toric variety, proving a claim of Bondal \cite[Corollary D]{HHL}. However, some algebraic aspects of their constructions are harder to determine.  For instance, if $F_\bullet$ is the free complex of $S$-modules corresponding to one of their resolutions, their work implies that the modules $H_i(F_\bullet)$ correspond to the zero sheaf on $X$ for all $i>0$, but it is not clear whether $H_i(F_\bullet)$ equals the zero module on the nose, i.e. it is not clear if $F_\bullet$ is acyclic as a complex of $S$-modules. The $S$-module that arises as $H_0(F_\bullet)$ is also unclear. By comparison, the complexes that arise in our construction are always acyclic, and they resolve normalizations of coordinate rings. However, we are not able to give as explicit of a description of the terms.  It would be very interesting to better compare these complexes, and to compare them with those in~\cite{BE, brown-sayrafi}.  Favero-Huang's approach~\cite{FH} can almost certainly yield all of the above results as well, and it would be interesting to compare to those resolutions too.

\begin{remark}
As our resolutions from Theorem~\ref{thm:virtual hochster} rely only on standard algebraic constructions, they can be directly computed in {\em Macaulay2} \cite{M2}.  The constructions in~\cite{HHL} are explicit, but due to their novelty, computing them in practice requires more effort.  Of course, if one could show that the two constructions coincide, this would shed more light on both.
\end{remark}
\section{Some elementary facts about toric varieties}
\begin{defn}
\label{def:ideal}
Let $X$, $Y$, and $S$ be as in Theorem~\ref{thm:virtual hochster}, $B \subseteq S$ the irrelevant ideal of $X$, and $Z$ the closure in $\Spec(S)$ of the inverse image of $Y$ under the canonical surjection $\pi \colon \Spec(S) \setminus V(B) \to X$. The \defi{defining ideal of $Y$ in $X$} is the radical ideal $I \subseteq S$ corresponding to the closed subset $Z \subseteq \Spec(S)$. 
\end{defn}

\begin{lemma}
\label{lem:technical}
Let $Z$ and $I$ be as in Definition~\ref{def:ideal}.
\begin{enumerate}
\item The irreducible components of $Z$ are affine toric varieties of the same dimension. Furthermore, if the divisor class group $\Cl(X)$ is torsion-free, then $Z$ is irreducible.
\item We have $\dim(S) - \dim(S/I) = \codim(Y \subseteq X)$.
\end{enumerate}
\end{lemma}

\begin{proof}
Since $Y \into X$ is a toric morphism, it induces an embedding $T_Y \into T_X$ on tori and hence a surjection $p \colon M_X \onto M_Y$ of lattices. Taking the pushout of the surjection $p$ and the canonical map $M_X \to \Z^{\dim{S}}$ yields the morphism
\begin{equation}
\label{eqn:ses}
\xymatrix{
0\ar[r]& M_X \ar[r]\ar[d]^-{p}& \ZZ^{\dim S}\ar[r]\ar[d]^{q} &\Cl(X)\ar[r]\ar[d]^{\cong}&0\\
0\ar[r]& M_Y \ar[r]& M' \ar[r]&\Cl(X)\ar[r]& 0
}
\end{equation}
of exact sequences. The abelian group $M'$ is isomorphic to $\Z^r \oplus A$, where $r$ is defined to be $\dim(S) - \dim(X) + \dim(Y)$, and $A$ is some finite abelian group. We observe that $I$ coincides with the radical of $J \coloneqq \ker(S \into k[\Z^{\dim{S}}] \xra{q} k[M'])$; note that $k[M']$ need not be reduced when $\on{char}(k) \ne 0$, since $M'$ may have torsion, and so $J$ need not be radical. Let us verify that $I = \on{rad}(J)$: since $p$ is surjective, the Snake Lemma implies that $q$ is surjective, and so $J$ is the defining ideal of the closure of $\Spec(k[M'])$ in $\Spec(S)$. Diagram~\eqref{eqn:ses} induces the following morphism of short exact sequences of algebraic groups:
$$
\xymatrix{
0 & \ar[l] T_X & \ar[l]_-{\a} \Spec(k[\Z^{\dim S}]) & \ar[l] \ker(\a) & \ar[l] 0 \\
0 & \ar[l] T_Y \ar[u] & \ar[l]_-{\b} \Spec(k[M']) \ar[u] & \ar[l]   \ker(\b) \ar[u]_{\cong}& \ar[l]  0.
}
$$
It follows that $\a^{-1}(T_Y)$ is equal to the image of $\Spec(k[M'])$ in $\Spec(k[\Z^{\dim S}])$. Since $Z$ is equal to the closure of $\a^{-1}(T_Y)$ in $\Spec(S)$, we conclude that $I = \on{rad}(J)$. 

Writing $R = k[\Z^r]$ and $A = \bigoplus_{i = 1}^t \Z / (n_i)$, we have $$
k[M'] \cong R[z_1, \dots, z_t] /(z_1^{n_1} -1, \dots, z_t^{n_t} - 1).
$$
The quotient of $k[M']$ by its nilradical is therefore a product of copies of $R$, and so $I$ is a finite intersection of prime ideals arising as kernels of ring homomorphisms $S \to R$. It therefore follows from \cite[Proposition 1.1.8]{CLS} that the irreducible components of $Z$ are affine toric varieties of dimension $r$. If $\Cl(X)$ is torsion-free, then the bottom row of Diagram~\eqref{eqn:ses} splits, and so $A = 0$, which means $I$ is prime. This proves (1). As for (2): we have shown that $\dim(Z) = r$, which is precisely $\dim(S) - \codim(Y \subseteq X)$. 
%Since $p$ is surjective, the Snake Lemma implies that $q$ is surjective as well. 
\end{proof}

\section{Examples}
%{\color{red}
%\begin{lemma}
%\label{lem:cox}
%In the setting of Theorem~\ref{thm:virtual hochster}, the quotient %$S/I$ is an affine semigroup ring. 
%\end{lemma}

%\begin{proof}
%Let $B$ denote the irrelevant ideal of $X$ and $Z$ the closure in $\Spec(S)$ of the inverse image of $Y$ under the canonical map $\pi \colon \Spec(S) \setminus V(B) \to X$. The affine variety $Z$ is the counterpart of $S/I$ under the usual ideal-variety correspondence. It thus suffices, by \cite[Proposition 1.1.8, Theorem 1.1.17]{CLS}, to show that $Z$ is the closure of the image of a map $T \to \Spec(S)$, where $T$ is a torus. Letting $S'$ and $B'$ denote the Cox ring and irrelevant ideal of $Y$, the inclusion $Y \into X$ is induced by a map $f \colon \Spec(S') \setminus V(B') \to \Spec(S) \setminus V(B)$; the inverse image of the torus in $Y$ under the composition $\pi f$ is the torus $T = (k^*)^{d}$, where $d = \dim(S')$. The map $f$ therefore induces a map $T \to \Spec(S)$; since $T$ is dense in $\Spec(S')$, the closure of the image of $T$ under $f$ is equal to $Z$.
%\end{proof}
%}
\begin{example}
\label{ex:Pn diagonal}
Let $X=\PP^n$ and $T=k[x_0,\dots, x_n,y_0,\dots,y_n]$, the Cox ring of $X\times X$.  Let $I_\Delta\subseteq T$ be the defining ideal (Definition~\ref{def:ideal}) of the diagonal $X\subseteq X\times X$, i.e. the ideal corresponding to the closure of the set of points in  $\Spec(T)$ of the form $(x_0, \dots, x_n, tx_0, \dots, tx_n)$, where $t\in k^*$. One easily checks that $I_\Delta$ is the kernel of the map $S \to  k[x_0, \dots, x_n, y_0, \dots, y_n, t]$ given by $x_i \mapsto x_i$ and $y_i \mapsto tx_i$, and so $T/I_\Delta$ is isomorphic to the normal semigroup ring $k[x_0, \dots, x_n, tx_0, \dots, tx_n]$. The ideal $I_\Delta$ is generated by the $2\times 2$ minors of the matrix
$
\begin{pmatrix}
x_0&x_1&\cdots &x_n\\
y_0&y_1&\cdots &y_n
\end{pmatrix}.
$
More specifically: these minors vanish on $\Delta$, and since this is a generic matrix, the ideal of $2\times 2$ minors is prime of codimension $n$.  As $T/I_\Delta$ is already normal, the virtual resolution of $T/I_\Delta$ arising from Theorem~\ref{thm:virtual hochster} is just the minimal free resolution of $T/I_\Delta$, which is given by the Eagon-Northcott complex on this matrix.
\end{example}

\begin{example}\label{ex:P112 diagonal}
Let $X$ be the weighted projective space $\PP(1,1,2)$ and $T$ the Cox ring $k[x_0,x_1,x_2,y_0,y_1,y_2]$ of $X\times X$.  
By a calculation similar to Example~\ref{ex:Pn diagonal}, the ring $T/I_\Delta$ is isomorphic to the semigroup ring
$
k[x_0,x_1,x_2,tx_0,tx_1,t^2x_2],
$
which is not normal because $tx_2$ lies in the fraction field and satisfies the integral equation $(tx_2)^2  - x_2 \cdot (t^2x_2)=0$.  Let $R$ be the normalization of $T/I_\Delta$. A presentation matrix for $R$ as a $T$-module is given as follows, where the rows correspond to the generators $1$ and $tx_2$:
\[
\bordermatrix{
&&&&& \cr
1 &x_{1}y_{0}-x_{0}y_{1}&x_{2}y_{0}&x_{2}y_{1}&x_{0}y_{2}&x_{1}y_{2}\cr
tx_2& 0&-x_{0}&-x_{1}&-y_{0}&-y_{1}
}.
\]

The free resolution of $R$ as a $T$-module is given by:
\begin{footnotesize}
\begin{equation}
\label{eq:res}
\begin{matrix} T \\ \oplus\\  T(-1,-1)\end{matrix} \xleftarrow{\left[\begin{smallmatrix} x_{1}y_{0}-x_{0}y_{1}&x_{2}y_{0}&x_{2}y_{1}&x_{0}y_{2}&x_{1}y_{2}\\
0&-x_{0}&-x_{1}&-y_{0}&-y_{1}
\end{smallmatrix}\right]} 
\begin{matrix}
T(-1,-1) \\
\oplus\\ 
T(-2,-1)^2 \\ \oplus \\
T(-1,-2)^2 
\end{matrix}
\xleftarrow{\left[\begin{smallmatrix} 
-x_{2}&0&-y_{2}\\
x_{1}&-y_{1}&0\\
-x_{0}&y_{0}&0\\
0&-x_{1}&-y_{1}\\
0&x_{0}&y_{0}
\end{smallmatrix}\right]}  
\begin{matrix}
T(-3,-1) \\
\oplus\\ 
T(-2,-2)\\ \oplus \\
T(-1,-3) 
\end{matrix}
 \gets 0.
\end{equation}
\end{footnotesize}

\noindent Additionally: we have the short exact sequence
$
0\to T/I_\Delta \to R \to Q \to 0,
$
and $Q = tx_2\cdot k[x_2,y_2]$.  One can directly compute that the sheaf $\widetilde{Q}$ corresponding to $Q$ is the zero sheaf on $X \times X$. In fact, since $Q$ is annihilated by $x_0,x_1,y_0$ and $y_1$, we can reduce to checking that $\widetilde{Q}$ is also zero on the affine patch $D(x_2y_2)$.  The global sections of $\widetilde{Q}$ on this patch are $Q[x_2^{-1},y_2^{-1}]_{(0,0)}=0$, and thus $\widetilde{Q}=0$ as desired.
%Namely, since $tx_2$ has degree $(1,1)$, $Q_{(a,b)}\ne 0$ if and only if $a,b\geq 0$ and $a,b$ are both odd.  But a $T$-module $M$ determines the zero sheaf on $X \times X$  if it is zero in degrees $(a,b)$ for all even integers $a,b \gg 0$; {\color{red} indeed, we need only check that $M[1/x_iy_j]_{(0,0)} = 0$ for $0 \le i, j \le 2$, and this is clear.} \michael{I couldn't find a citation so I just wrote a quick proof (one just needs to check that the module is zero on the canonical affine open cover, and this is clear). This previously said that $M$ determines the zero sheaf on $X \times X$ \emph{if and only if}...etc. That is true, but we only need one direction, so it's quicker to just write this. But we can add a proof of the other direction if you prefer.}
\end{example}

\begin{remark}\label{rmk:not line bundles}
Since $\OO(-1)$ and $\OO(-3)$ are not vector bundles on $\PP(1, 1, 2)$, the resolution~\eqref{eq:res} does not induce a locally free resolution of the diagonal. Indeed, virtual resolutions are not guaranteed to induce locally free resolutions of $\OO_X$-modules unless $X$ is smooth. Alternatively, as in~\cite{HHL}, one could consider the corresponding toric stack.
\end{remark}

\begin{remark}
In many of the prior known cases of Theorem~\ref{thm:virtual syzygy}, a slightly stronger result was proven.  Namely, it was shown that for any such $M$, there exists another module $M'$ satisfying $\widetilde{M}=\widetilde{M'}$ and $\pdim(M')\leq \dim X$; see~\cite{EES,bruce-heller-sayrafi,yang}.  It would be interesting to determine if this was true in general. 
%The fact that the virtual resolution of the diagonal from Theorem~\ref{thm:virtual diagonal} has this property makes it seem reasonable to hope that this might hold.
\end{remark}

\section*{Acknowledgments}  We are very grateful to Andrew Hanlon, Jeff Hicks, and Oleg Lazarev for patiently talking to us about their work and for several inspiring conversations. We only found this approach because of our efforts to understand their beautiful results. We also thank Christine Berkesch, Lauren Cranton Heller, Mahrud Sayrafi, and Jay Yang for helpful comments and discussions. Finally, we thank the anonymous referee for many helpful suggestions.

\bibliographystyle{amsalpha}
\bibliography{Bibliography}

\begin{bibdiv}
\begin{biblist}

\bib{anderson}{article}{
      author={Anderson, Reginald},
       title={A resolution of the diagonal for toric {D}eligne-{M}umford
  stacks},
        date={2023},
     journal={arXiv:2303.17497},
}

\bib{BC}{article}{
      author={Bai, Shaoyun},
      author={C\^{o}t\'{e}, Laurent},
       title={On the {R}ouquier dimension of wrapped {F}ukaya categories and a
  conjecture of {O}rlov},
        date={2023},
        ISSN={0010-437X},
     journal={Compos. Math.},
      volume={159},
      number={3},
       pages={437\ndash 487},
         url={https://doi.org/10.1112/S0010437X22007886},
      review={\MR{4556220}},
}

\bib{BDM}{article}{
      author={Ballard, Matthew~R.},
      author={Duncan, Alexander},
      author={McFaddin, Patrick~K.},
       title={The toric {F}robenius morphism and a conjecture of {O}rlov},
        date={2019},
        ISSN={2199-675X},
     journal={Eur. J. Math.},
      volume={5},
      number={3},
       pages={640\ndash 645},
         url={https://doi.org/10.1007/s40879-018-0266-5},
      review={\MR{3993255}},
}

\bib{BE}{article}{
      author={Brown, Michael~K.},
      author={Erman, Daniel},
       title={Tate resolutions on toric varieties},
        date={2021},
     journal={to appear in the Journal of the European Mathematical Society
  (JEMS)},
}

\bib{BES}{article}{
      author={Berkesch, Christine},
      author={Erman, Daniel},
      author={Smith, Gregory~G.},
       title={Virtual resolutions for a product of projective spaces},
        date={2020},
        ISSN={2313-1691},
     journal={Algebr. Geom.},
      volume={7},
      number={4},
       pages={460\ndash 481},
}

\bib{BF}{article}{
      author={Ballard, Matthew},
      author={Favero, David},
       title={Hochschild dimensions of tilting objects},
        date={2012},
        ISSN={1073-7928},
     journal={Int. Math. Res. Not. IMRN},
      number={11},
       pages={2607\ndash 2645},
         url={https://doi.org/10.1093/imrn/rnr124},
      review={\MR{2926991}},
}

\bib{BFK}{article}{
      author={Ballard, Matthew},
      author={Favero, David},
      author={Katzarkov, Ludmil},
       title={Variation of geometric invariant theory quotients and derived
  categories},
        date={2019},
        ISSN={0075-4102},
     journal={J. Reine Angew. Math.},
      volume={746},
       pages={235\ndash 303},
         url={https://doi.org/10.1515/crelle-2015-0096},
      review={\MR{3895631}},
}

\bib{BH}{book}{
      author={Bruns, Winfried},
      author={Herzog, J\"{u}rgen},
       title={Cohen-{M}acaulay rings},
      series={Cambridge Studies in Advanced Mathematics},
   publisher={Cambridge University Press, Cambridge},
        date={1993},
      volume={39},
        ISBN={0-521-41068-1},
      review={\MR{1251956}},
}

\bib{bruce-heller-sayrafi}{misc}{
      author={Bruce, Juliette},
      author={Heller, Lauren~Cranton},
      author={Sayrafi, Mahrud},
       title={Characterizing multigraded regularity on products of projective
  spaces},
        date={2021},
        note={arXiv:2110.10705},
}

\bib{berkesch-klein-loper-yang}{article}{
      author={Berkesch, Christine},
      author={Klein, Patricia},
      author={Loper, Michael~C.},
      author={Yang, Jay},
       title={Combinatorial aspects of virtually {C}ohen-{M}acaulay sheaves},
        date={2021},
     journal={S\'{e}m. Lothar. Combin.},
      volume={85B},
       pages={Art. 63, 13},
}

\bib{borisov}{article}{
      author={Borisov, Lev~A.},
       title={String cohomology of a toroidal singularity},
        date={2000},
        ISSN={1056-3911},
     journal={J. Algebraic Geom.},
      volume={9},
      number={2},
       pages={289\ndash 300},
      review={\MR{1735773}},
}

\bib{BPS}{article}{
      author={Bayer, Dave},
      author={Popescu, Sorin},
      author={Sturmfels, Bernd},
       title={Syzygies of unimodular {L}awrence ideals},
        date={2001},
        ISSN={0075-4102},
     journal={J. Reine Angew. Math.},
      volume={534},
       pages={169\ndash 186},
         url={https://doi-org.spot.lib.auburn.edu/10.1515/crll.2001.040},
      review={\MR{1831636}},
}

\bib{brown-sayrafi}{article}{
      author={Brown, Michael~K.},
      author={Sayrafi, Mahrud},
       title={A short resolution of the diagonal for smooth projective toric
  varieties of {P}icard rank 2},
        date={2022},
     journal={to appear in Algebra \& Number Theory},
}

\bib{canonaco}{article}{
      author={Canonaco, Alberto},
       title={Beilinson resolutions on weighted projective spaces},
        date={2003},
        ISSN={1631-073X},
     journal={C. R. Math. Acad. Sci. Paris},
      volume={336},
      number={1},
       pages={35\ndash 40},
         url={https://doi.org/10.1016/S1631-073X(02)00004-3},
      review={\MR{1968899}},
}

\bib{CLS}{book}{
      author={Cox, David~A.},
      author={Little, John~B.},
      author={Schenck, Henry~K.},
       title={Toric varieties},
   publisher={American Mathematical Soc.},
        date={2011},
      volume={124},
}

\bib{EES}{article}{
      author={Eisenbud, David},
      author={Erman, Daniel},
      author={Schreyer, Frank-Olaf},
       title={Tate resolutions for products of projective spaces},
        date={2015},
     journal={Acta Mathematica Vietnamica},
      volume={40},
      number={1},
       pages={5\ndash 36},
}

\bib{eisenbudbook}{book}{
      author={Eisenbud, David},
       title={Commutative algebra},
      series={Graduate Texts in Mathematics},
   publisher={Springer-Verlag, New York},
        date={1995},
      volume={150},
        ISBN={0-387-94268-8; 0-387-94269-6},
         url={https://doi.org/10.1007/978-1-4612-5350-1},
        note={With a view toward algebraic geometry},
      review={\MR{1322960}},
}

\bib{FH}{article}{
      author={Favero, David},
      author={Huang, Jesse},
       title={Rouquier dimension is {K}rull dimension for normal toric
  varieties},
        date={2023},
     journal={arXiv preprint arXiv:2302.09158},
}

\bib{HHL}{misc}{
      author={Hanlon, Andrew},
      author={Hicks, Jeff},
      author={Lazarev, Oleg},
       title={Resolutions of toric subvarieties by line bundles and
  applications},
        date={2023},
        note={arXiv:2303.03763},
}

\bib{HNV}{article}{
      author={Harada, Megumi},
      author={Nowroozi, Maryam},
      author={Van~Tuyl, Adam},
       title={Virtual resolutions of points in {$\Bbb P^1\times\Bbb P^1$}},
        date={2022},
        ISSN={0022-4049},
     journal={J. Pure Appl. Algebra},
      volume={226},
      number={12},
       pages={Paper No. 107140, 18},
         url={https://doi.org/10.1016/j.jpaa.2022.107140},
      review={\MR{4438917}},
}

\bib{hochster}{article}{
      author={Hochster, M.},
       title={Rings of invariants of tori, {C}ohen-{M}acaulay rings generated
  by monomials, and polytopes},
        date={1972},
     journal={Ann. of Math. (2)},
      volume={96},
       pages={318\ndash 337},
}

\bib{M2}{misc}{
       title={Macaulay2 – a system for computation in algebraic geometry and
  commutative algebra programmed by {D}. {G}rayson and {M}. {S}tillman},
         url={http://www.math.uiuc.edu/Macaulay2/},
}

\bib{orlov}{article}{
      author={Orlov, Dmitri},
       title={Remarks on generators and dimensions of triangulated categories},
        date={2009},
        ISSN={1609-3321},
     journal={Mosc. Math. J.},
      volume={9},
      number={1},
       pages={153\ndash 159, back matter},
         url={https://doi.org/10.17323/1609-4514-2009-9-1-143-149},
      review={\MR{2567400}},
}

\bib{yang}{article}{
      author={Yang, Jay},
       title={Virtual resolutions of monomial ideals on toric varieties},
        date={2021},
     journal={Proc. Amer. Math. Soc. Ser. B},
      volume={8},
       pages={100\ndash 111},
}

\end{biblist}
\end{bibdiv}

\Addresses

\end{document}